\documentclass[12pt]{amsart}
\usepackage[utf8]{inputenc}
\usepackage{t1enc}
\usepackage[english]{babel}
\usepackage{amssymb}
\usepackage{amsthm}
\usepackage{amsmath}
\usepackage{verbatim}
\usepackage{graphicx}
\usepackage{pgf,tikz,pgfplots}
\pgfplotsset{compat=1.15}
\usepackage{mathrsfs}
\usetikzlibrary{arrows}
\usepackage{caption}
\usepackage{hyperref}
\usepackage{enumitem}
\usepackage{ragged2e}

\theoremstyle{definition}
\newtheorem{mydef}{Definition}[section]

\theoremstyle{plain}
\newtheorem{thm}[mydef]{Theorem}

\newtheorem{all}[mydef]{Claim}

\newtheorem{kov}[mydef]{Corollary}

\newtheorem{mj}[mydef]{Remark}
\newtheorem{sej}[mydef]{Conjecture}
\newtheorem*{jel}{Notation}

\newcommand\cF{{\mathcal F}}

\title{A new intersection condition in extremal set theory}

\author{Kartal Nagy} 
\address{E{\"o}tv{\"o}s Lor\'and University, Budapest, Hungary}
\email{kartal97@student.elte.hu}

\begin{document}

\maketitle

\begin{abstract}
We call a family $\mathcal{F}$ $(3,2,\ell)$-intersecting if $|A \cap B|+|B \cap C|+|C \cap A| \geq \ell$ for all $A$, $B$, $C \in \mathcal{F}$. We try to look for the maximum size of such a family $\mathcal{F}$ in case when $\mathcal{F} \subset {[n] \choose k}$ or $\mathcal{F} \subset 2^{[n]}$. In the uniform case we show that if $\mathcal{F}$ is $(3,2,2)$-intersecting, then $\vert \mathcal{F} \vert \leq {n+1 \choose k-1}+{n \choose k-2}$ and if $\mathcal{F}$ is $(3,2,3)$-intersecting, then $|\mathcal{F}| \leq {n \choose k-1} + 2 {n \choose k-3} + 3 {n-1 \choose k-3}$. For the lower bound we construct a $(3,2,\ell)$-intersecting family and we show that this bound is sharp when $\ell=2$ or $3$ and $n$ is sufficiently large compared to $k$. In the non-uniform case we give an upper bound for a $(3,2,n-x)$-intersecting family, when $n$ is sufficiently large compared to $x$. 
\end{abstract}

\section{Introduction}

We denote the set $\{ 1,2,\dots,n \}$ by $[n]$, the family of the subsets of $[n]$ by $2^{[n]}$, and the family of the $i$-element subsets of $[n]$ by $\displaystyle\binom{[n]}{i}$. If $i > n$, then this notation means the empty set.
\vspace{1mm}

We will mostly use the terminology \textit{families} for collections of sets, although \textit{set systems} is also widely used in the literature. If all sets in a family $\cF$ are of the same size, we say $\cF$ is \textit{uniform}.
Otherwise $\cF$ is called \textit{non-uniform}.

\begin{mydef}
    A family $\cF$ is said to be intersecting if $|A \cap B| \neq 0$ holds for all $A$, $B \in \cF$.
\end{mydef}

The following two theorems of Erdős, Ko and Rado are central ones in extremal set theory. They determine the largest size of a family containing no disjoint pairs of sets. The first theorem solves the case of non-unifom families, the second one determines the maximum for the uniform case.

\begin{thm} {\rm \cite{EKR}}
Let $\cF\subset 2^{[n]}$ be an intersecting family. Then $|\cF|\le 2^{n-1}$.
\end{thm}

\begin{thm} {\rm \cite{EKR}} \label{1metszet}
Let $k$, $n \geq 2k$ be integers and $\cF\subset {[n] \choose k}$ be an intersecting family.  Then $|\cF|\le \displaystyle\binom{n-1}{k-1}$.
\end{thm}

The concept "intersecting family" has several generalizations, as follows.

\begin{mydef}
A family $\mathcal{F}$ is said to be \textit{$t$-intersecting} if $|A \cap B| \geq t$ holds for all $A$, $B \in \mathcal{F}$.
\end{mydef}

\begin{mydef}
A family $\mathcal{F}$ is said to be \textit{trivially $t$-intersecting} if there exists a $t$-set $T$ such that $T \subset F$ holds for all $F \in \mathcal{F}$, and \textit{non-trivially $t$-intersecting} if it is $t$-intersecting but $|\bigcap_{F \in \mathcal{F}} F| < t$.
\end{mydef}

It was proved by Erdős, Ko, and Rado in their seminal paper \cite{EKR} that for any $1 \leq t < k$, the largest $t$-intersecting families $\mathcal{F} \subset {[n] \choose k}$ are the maximal trivially $t$-intersecting ones if $n$ is large enough. The threshold function $n_0(k,t) = (k-t+1)(t+1)$ of this property was determined by  Frankl \cite{Frankl3} for $t \geq 15$. Later Wilson \cite{Wilson} filled in the gap $2 \leq t \leq 14$.

\begin{thm} \label{tmetszet}
Trivially $t$-intersecting family is the maximal $t$-inter\-secting family if and only if $n \geq (t+1)(k-t+1)$.
\end{thm}

The non-uniform version of the theorem for $t \geq 2$ was proved by Katona.

\begin{thm} {\rm \cite{Katona}}
Let $\cF\subset 2^{[n]}$ be a $t$-intersecting family with $1 \leq t \leq n$. Then
\begin{displaymath}
|\mathcal{F}| \leq
\left\{ \begin{array}{c c}
\sum_{i=\frac{n+t}{2}}^n {n \choose i} & \text{if $n+t$ is even} \vspace{1mm} \\ \vspace{1mm}
\sum_{i=\frac{n+t+1}{2}}^n {n \choose i} + {n-1 \choose \frac{n+t-1}{2}} & \text{if $n+t$ is odd}
\end{array}
\right.
\end{displaymath}
\end{thm}

One of the famous questions of this area is the Erdős Matching Conjecture (EMC), in which the maximum number of $k$-element subsets should be determined under the condition that no $s+1$ of them are pairwise disjoint. This value is denoted by $m(n,k,s)$.

\begin{sej} \label{mnks} {\rm \cite{ErdosMC}}
    For $n \geq k(s+1)$, $$m(n,k,s) =\max \left\{ {n \choose k} - {n-s \choose k},{k(s+1)-1 \choose k} \right\}.$$
\end{sej}

It was one of the favourite problems of Erdős and there was hardly a combinatorial lecture of him where he did not mention it. 
The EMC is trivial for $k=1$ and was proved by Erdős and Gallai \cite{ErdosGallai} for $k=2$ then it was settled in the case $k=3$ and $n \geq 4s$ in \cite{FranklRodl}, for $k=3$, all $n$ and $s \geq s_0$ in \cite{Luczak}, and finally, it was completely resolved for $k=3$ in \cite{Frankl4}.

The case $s=1$ is the classical Erdős-Ko-Rado theorem. In his original paper, Erdős proved \ref{mnks} for $n \geq n_0 (k,s)$ for some $n_0 (k,s)$. His result was sharpened by Bollobás, Daykin and Erdős \cite{Bollobas}, they proved for $n \geq 2k^3s$. Subsequently, Hao, Loh and Sudakov \cite{Huang} proved the EMC for $n \geq 3k^2s$.

Frankl \cite{Frankl5} showed that $m(n,k,s) \leq  {n \choose k} - {n -s \choose k}$ when $n \geq (2s+1)k-s$. For $s$ very large Frankl and Kupavskii prove the following theorem.

\begin{thm} {\rm \cite{FranklKup}}
There exists an absolute constant $s_0$, such that $$m(n,k,s)= {n \choose k} - {n -s \choose k}$$ holds if $n \geq \frac{5}{3}sk-\frac{2}{2}s$ and $s \geq s_0$.
\end{thm}

In the non-uniform case Kleitman considered when $s=2$. 

\begin{thm} {\rm \cite{Kleitman}}
Let $n=3m+2$ and $\mathcal{F} \subset 2^{[n]}$ be a family with no three of whose members are pairwise disjoint. Then $$|\mathcal{F}| \leq \sum_{j = m+1}^{3m+2} {3m+2 \choose j}$$ is the best upper bound.
\end{thm}

\begin{thm} {\rm \cite{Kleitman}}
Let $n=3m$ and $\mathcal{F} \subset 2^{[n]}$ be a family with no three of whose members are pairwise disjoint. Then $$|\mathcal{F}| \leq \sum_{j = m+1}^{3m} {3m \choose j} + \frac{2}{3}{3m \choose m}$$ is the best upper bound.
\end{thm}

\section{Results}

The condition of Kleitman's theorem can be written in the following form: $\vert A \cap B \vert + \vert B \cap C \vert + \vert C \cap A \vert \geq 1$ holds for any 3 distinct members $A$, $B$, $C \in \cF$. The main goal of the present paper is to study the case when 1 is replaced by $\ell$. Since the case $\ell=1$ has been studied a lot, we will only look at the case $\ell \geq 2$. 

Let $d: (2^{[n]})^3 \rightarrow \mathbb{N}$ be a function such that for all $A$, $B$, $C \in 2^{[n]}$, $d(A,B,C)= \vert A \cap B \vert + \vert B \cap C \vert + \vert C \cap A \vert$. Denote by $\mathbf{H}_\ell$ the class of families such that $\mathcal{H}$, $d(A,B,C) \geq \ell$ for all sets $A$, $B$, $C \in \mathcal{H}$. Let $g(n,k,\ell) = \max \{ |\mathcal{F}| : \mathcal{F} \subset {[n] \choose k}, \mathcal{F}\in \mathbf{H_{\ell}}\}$ and $h(n,\ell) = \max \{ |\mathcal{F}| : \mathcal{F} \subset 2^{[n]}, \mathcal{F}\in \mathbf{H_{\ell}}\}$. 

First we consider the uniform case. We prove the following upper and lower bounds:

\begin{thm}
\label{uni1}
Let $n \geq 3k-2 \geq 4$ be positive integers. Then 
$$g(n,k,2) \leq {n+1 \choose k-1}+{n \choose k-2}.$$
\end{thm}

\begin{thm}
\label{uni2}
Let $n \geq 3k-3 \geq 3$ be positive integers. Then 
$$ g(n,k,3) \leq {n \choose k-1} + 2 {n \choose k-3} + 3 {n-1 \choose k-3}.$$
\end{thm}


\begin{thm} \label{unialso}
Let $k$, $\ell$, $n$ be positive integers with $3k \geq \ell$. Then 
\begin{equation} \label{gnkl}
g(n,k,\ell) \geq \max_{k \geq j \geq \frac{\ell}{3}} \left\{ \sum_{i=j}^k {f(j,\ell) \choose i}{n-f(j,\ell) \choose k-i} \right\},  
\end{equation}
where 
\begin{displaymath}
f(j, \ell) = 
\left\{ \begin{array}{l l}
2j- \lceil \frac{\ell}{3} \rceil & \text{if} \hspace{2mm} \frac{\ell}{3} \leq j < \frac{2\ell}{3}, \vspace{1mm} \\
3j-\ell & \text{if} \hspace{2mm} \frac{2\ell}{3} \leq j.\\
\end{array}
\right.
\end{displaymath}
\end{thm}

\begin{mj}
If $a<b$, then we interpret ${a \choose b}$ as zero.
\end{mj}

Finally, we show that if $\ell=2$ or $3$ and $n$ is sufficiently large compared to $k$, then our upper bounds in Theorem \ref{uni1} and Theorem \ref{uni2} can be improved, reaching the value of the lower bound in Theorem \ref{unialso} for $j=1$.

\begin{thm}
\label{unialso2}
Let $n \geq 4 \cdot k^3$ be positive integers. Then $g(n,k,2)=g(n,k,3)={n-1 \choose k-1}$.
\end{thm}

This theorem is a sharpening of the Erdős-Ko-Rado theorem, since it achieves the same upper bound for the corresponding pairs $(n,k)$ under weaker conditions. 

In the non-uniform case, we investigate families $\mathcal{F}$ for which \break $d(A,B,C) \geq 3n-x$ for all $A$, $B$, $C \in \mathcal{F}$. In the case when $x=6p$ ($p$ is an integer) the following construction seems to be optimal: we select all sets with at least $n-p$ elements, then for any set of triples $d(A,B,C) \geq 3n-6p$. If $x=6p$, this will be the optimal solution for large $n$ and if $6p<x<6(p+1)$, then the optimal solution contains all $n-p$-element sets and some of the $n-p-1$-element sets.

\begin{thm} \label{3n-p}
Let $x=6p+q$, where $1 \leq p$,  $0 \leq q \leq 5$ and $n \geq 2^{3p+2} p^2 + p + 1$. Then

\begin{displaymath}
h(n,3n-x) =
\left\{ \begin{array}{c c}
\sum_{i=0}^p {n \choose i} & q=0 \ or \ 1, \vspace{1mm} \\ \vspace{1mm}
\sum_{i=0}^p {n \choose i} + {n-2 \choose p-1} & q=2,\\ \vspace{1mm}
\sum_{i=0}^p {n \choose i} + {n-1 \choose p} & q=3 \ or \ 4,\\ \vspace{1mm}
\sum_{i=0}^p {n \choose i} + {n-1 \choose p} + {n-2 \choose p} & q=5.\\
\end{array}
\right.
\end{displaymath}
\end{thm}

\section{The uniform case}

In the next section, we use Frankl's shifting technique. \cite{Frankl6}

\subsection{Shifting a family}

\begin{jel}
Let $\mathcal{A}$ be a three-member family $\mathcal{A}=\{A,B,C\}$.
\end{jel}

\begin{mydef}
Let $1 \leq y < x \leq n$ be two elements of $[n]$ and let $\mathcal{F} \subset 2^{[n]}$ be a family. Then
$$\mathcal{F}' =\{F:F \in \mathcal{F}, (x \notin F) \vee (x,y \in F) \vee (x \in F, y \notin F, F \cup \{y \}- \{ x \} \in \mathcal{F}) \}$$
and
$$\mathcal{F}'' =\{ F:F \in \mathcal{F}, x \in F, y \notin F, F \cup \{y \} - \{ x \} \notin \mathcal{F} \}.$$
\end{mydef}

\begin{mydef}
Let $x$, $y$ and $\mathcal{F}$ be same as above. Then $$\tau_{x,y}(\mathcal{F}) = \mathcal{F}' \cup \{ F \cup \{ y\} - \{ x \}: F \in \mathcal{F}'' \}. $$ This operation is called a shifting of $\mathcal{F}$ (from $x$ to $y$).
\end{mydef}

\begin{mydef}
Let $\tau_{x,y}: \mathcal{F} \rightarrow \tau_{x,y}(\mathcal{F})$ be a bijection such that for $\forall F \in \mathcal{F}$
\begin{equation}
    \tau_{x,y}(F)=
    \begin{cases}
        F & \text{if $(x, y \in F) \vee (x,y \notin F) $} \\
        F & \text{if $x \in F, y \notin F, F \cup \{y \}- \{ x \} \in \mathcal{F}$ }\\
        F - \{ x \}  \cup \{ y \} & \text{otherwise}
    \end{cases}       
\end{equation}
\end{mydef}

\begin{all}
\label{tolas}
Let $\mathcal{F} \in 2^{[n]}$ be a family such that $d(\mathcal{A}) \geq \ell$ for all $A$, $B$, $C \in \mathcal{F}$. Then  $d(\tau_{x,y}(A),\tau_{x,y}(B),\tau_{x,y}(C)) \geq \ell$.
\end{all}

\begin{jel}
When it does not cause ambiguity we delete the indices $x$, $y$ from the notation $\tau_{x,y}$.
\end{jel}

\begin{proof}
We prove this claim by contradiction. Suppose that the conditions are true, but there exist sets $A$, $B$ and $C \in \cF$ such that $d(\tau(\mathcal{A})) < \ell$.

Suppose that $z \notin \{x,y\}$. Then $z \in F$ holds if and only if $z \in \tau(F)$ so 
$d(\mathcal{A} \setminus\{x,y\}) = d(\tau(\mathcal{A}) \setminus \{x,y \}) = c$.

Let $\cF^1 = \{ H: H \in  \cF, H = \tau(H)\}$ and $\cF^2 = \{ H: H \in  \cF, H \neq \tau(H)\}$. Let $\mathcal{A}^1=\mathcal{A} \cap \cF^1$ and $\mathcal{A}^2=\mathcal{A} \cap \cF^2$. We can assume that $\mathcal{A}^1$ contains the element $x$ $p$ times, the element $y$ $q$ times and $\mathcal{A}^2$ contains the element $x$ $r$ times. (It is obvious that $\mathcal{A}^2$ contain the element $y$ $0$ times.) Then $d(\mathcal{A}) = c + {p+r \choose 2}+{q \choose 2} > c + {p \choose 2}+{q+r \choose 2} = d(\tau(\mathcal{A}))$.

This is true only if $p+r > p > q$ (and $p+r > q+r > q$). This means that there are at least $p-q$ sets $H \in \mathcal{A}^1$ such that $x \in H$, $y \notin H$. Since $x \in \tau(H)$ and $y \notin \tau(H)$, then $H' = H - \{x\} \cup \{y\} \in \cF$. So if $\mathcal{A}'$ is a family such that we replace $x$ with $y$ in this $p-q$ sets in $\mathcal{A}^1$, then $\mathcal{A}' \subset \cF$ and $d(\mathcal{A}') = c + {p \choose 2}+{q+r \choose 2} = d(\tau(\mathcal{A})) < \ell$, a contradiction. \end{proof}

\begin{mydef}
Let $H \in 2^{[n]}$ and $\mathcal{H} \subset 2^{[n]}$, then $s(H)=\sum_{i \in H} i$ and $s(\mathcal{H})=\sum_{H \in \mathcal{H}} s(H).$
\end{mydef}

\begin{all} \label{csokken}
If $\tau(\mathcal{H}) \neq \mathcal{H}$, then $s(\tau(\mathcal{H})) < s(\mathcal{H})$. $\square$
\end{all} 

\begin{mydef}
Let $H= \{ h_1, \dots, h_k \}$ and $I= \{ i_1 \dots i_k \} \in {[n] \choose k}$ be sets such that $h_1 < h_2 < \dots < h_k$ and $i_1 < i_2 < \dots < i_k$. We write that $I \leq H$, if $i_j \leq h_j$ holds for all $1 \leq j \leq k$. 
\end{mydef}

\begin{mydef}
A uniform family $\mathcal{F} \in {[n] \choose k}$ is called \textit{shifted}, if $H \in \mathcal{F}$ and $I \leq H$ imply $I \in \mathcal{F}$. A non-uniform family $\mathcal{F} \in 2^{[n]}$ is called \textit{shifted}, if $\mathcal{F} \cap {[n] \choose i}$ is shifted for every $0 \leq i \leq n$.
\end{mydef}

\begin{all}
For a family $\mathcal{F} \in 2^{[n]}$, there exist $x$ and $y$ such that $s(\mathcal{F}) > s(\tau_{x,y}(\mathcal{F}))$ if and only if $\mathcal{F}$ is not shifted. $\square$
\end{all}

\begin{all}
\label{unibalra}
For a family $\mathcal{F} \in 2^{[n]}$, there exist $\tau_1,\dots,\tau_k$ such that $\mathcal{F'}=\tau_{x_1,y_1}(\tau_{x_2,y_2}(\dots (\tau_{x_k,y_k}(\mathcal{F}))))$ is a shifted family.
\end{all}

\begin{proof}
If $\mathcal{F}$ is not shifted, there exists a $\tau$ shifting such that $s(\mathcal{F}) > s(\tau(\mathcal{F}))$. Since $s(\mathcal{H}) \in \mathbb{N}$, by Claim \ref{csokken} there cannot be an infinite decreasing sequence, so after a finite number of shifting we obtain a shifted family.
\end{proof}

\begin{kov}
\label{balra}
As a result of Claim \ref{tolas} and \ref{unibalra}, if a family $\mathcal{F} \in \mathbf{H}_\ell$, then there exists a shifted family $\mathcal{F'} \in \mathbf{H}_\ell$ such that $|\mathcal{F}|=|\mathcal{F}'|$.
Therefore it is sufficient to consider the shifted families, when we search for the maximum sized family in $\mathbf{H}_\ell$.
\end{kov}

\begin{all}
\label{3k-2}
Let $\ell=2$ or $3$, $\mathcal{F} \subset {[n] \choose k}$ and $\mathcal{F} \in \mathbf{H}_\ell$ be a shifted family. If $A$, $B$, $C \in \mathcal{F}$ and $A'=A \cap [3k-\ell]$, $B'=B \cap [3k-\ell]$ and $C'=C \cap [3k-\ell]$, then $d(A',B',C') \geq \ell$, that is $\mathcal{F} \cap [3k-\ell] \in \mathcal{H}_\ell$. 
\end{all}

\begin{proof}

Let $A_1=A \setminus [3k-\ell]$, $B_1=B \setminus [3k-\ell]$ and $C_1=C \setminus [3k-\ell]$. 

Suppose that the statement is not true and let $A$, $B$, $C \in \mathcal{F}$ such that $d(A',B',C') < \ell$ and $m=|A_1|+|B_1|+|C_1|$ is the smallest possible. Since $d(A_1, B_1, C_1) > 0$, so $m \geq 2$ and wlog $p \in A_1$.

If there is an $a \in [3k-\ell] \setminus (A' \cup B' \cup C')$ then $\overline{A}=A-\{p\} \cup \{a\} \in \mathcal{F}$, $d(\overline{A}', B', C')= d(A', B', C') < \ell$ and $|\overline{A}_1|+|B_1|+|C_1| < m$ which is impossible.

So $A' \cup B' \cup C' = [3k-\ell]$. Since $d(A',B',C') < \ell \leq 3$, then there is no $a \in [3k-\ell]$ such that $a \in A \cap B \cap C$. Since $|A'|+|B'|+|C'| = 3k-m \leq 3k-2$, then $d(A',B',C') = \ell-m \leq \ell - 2$. Let $q \in [3k-\ell]$ so that $q$ is in one of $A'$, $B'$, $C'$. Then $\overline{A} = A-\{p\} \cup \{q\} \in F$ and $d(\overline{A'},B',C' ) \leq \ell-2+1 <\ell$ but $|\overline{A}_1|+|B_1|+|C_1| < m$ which is a contradiction. \end{proof}

\begin{mj}
Claim \ref{3k-2} is not true if $\ell \geq 4$. For example let $\ell=4$, $k=3$, $A=\{1,2,6\}$, $B=\{1,3,6\}$, $C=\{1,4,5\}$ and $H \in \mathcal{F}$ if and only if $H \leq A$ or $H \leq B$ or $H \leq C$. Then $\mathcal{F} \in \mathbf{H}_\ell$, because for all $H \in \cF$ are true that $1 \in H$ and contains two element of $\{2,3,4,5,6\}$, so $d(\mathcal{A} \cap {1})=3$ and $d(\mathcal{A} - {1}) \geq 1$, so $d(\mathcal{A}) \geq 4$ but $3k-\ell=5$ and $d(A',B',C') < 4$.
\end{mj}

\begin{all}
\label{rekurzio}
If $\ell=2$ or $\ell=3$, $k \geq 3$ and $n>3k-\ell$, then $g(n,k,\ell) \leq g(n-1,k,\ell)+g(n-1,k-1,\ell)$.
\end{all}

\begin{proof}
Let $\mathcal{M}_{n,k,\ell} \subset {[n] \choose k}$ be a shifted family such that $\mathcal{M}_{n,k,\ell} \in \mathbf{H}_\ell$ and $|\mathcal{M}_{n,k,\ell}|=g(n,k,\ell)$. This exists because of Corollary \ref{balra}.

Let $\mathcal{M}^1_{n,k,\ell} \subset {[n-1] \choose k}$ be the family $\mathcal{M}^1_{n,k,\ell} = \{ M : M \in \mathcal{M}_{n,k,\ell}, \{n\} \notin M \}$ and let  $\mathcal{M}^2_{n,k,\ell} \subset {[n-1] \choose k-1}$ be the family such that $\mathcal{M}^2_{n,k,\ell} = \{ M \setminus \{n\} : M \in \mathcal{M}_{n,k,\ell}, n \in M \}$.

It is easy to check that $\mathcal{M}^1_{n,k,\ell}$ and $\mathcal{M}^2_{n,k,\ell}$ are shifted families. Moreover $\mathcal{M}^1_{n,k,\ell} \subset \mathcal{M}_{n,k,\ell}$, so $\mathcal{M}^1_{n,k,\ell} \in \mathbf{H}_\ell$. Since $n > 3k-\ell$, $\mathcal{M}^2_{n,k,\ell} \cap [3k-\ell] \subset \mathcal{M}_{n,k,\ell} \cap [3k-\ell]$. Using Claim \ref{3k-2}, we obtain that $\mathcal{M}_{n,k,\ell} \in \mathbf{H}_\ell$ implies $\mathcal{M}^2_{n,k,\ell} \in \mathbf{H}_\ell$. 

Since $|\mathcal{M}^1_{n,k,\ell}| \leq g(n-1,k,\ell)$ and $|\mathcal{M}^2_{n,k,\ell}| \leq g(n-1,k-1,\ell)$, so $g(n,k,\ell)=|\mathcal{M}_{n,k,\ell}|=|\mathcal{M}^1_{n,k,\ell}|+|\mathcal{M}^2_{n,k,\ell}| \leq g(n-1,k,\ell) + g(n-1,k-1,\ell)$.
\end{proof}

\subsection{The proof of Theorem \ref{uni1} and \ref{uni2}}

\begin{proof}

It is easy to see that $g(3k-\ell,k,\ell) = {3k - \ell \choose k}$. In the next step we show that Theorem \ref{uni1} and Theorem \ref{uni2} are true when $k=2$.

\begin{all}
$g(n,2,2) \leq {n+1 \choose 2-1} + {n \choose 2-2} = n+2$.
\end{all}

\begin{proof}
The statement is trivially true when $n \leq 3$, so we assume that $n \geq 4$. If there are $A$ and $B$ in $\mathcal{F}$ such that $|A \cap B| = 0$, then for all $C \in \mathcal{F}$, $C \subset A \cup B$. Since $|A \cup B|=4$, then $|\mathcal{F}| \leq 6 \leq n+2$. Therefore, we can assume that such $A$ and $B$ do not exist. So $A = \{a,b\}$ and $B= \{a,c\}$. Then, for all $C \in \mathcal{F}$, either $a \in C$ or ${b,c} \in C$. Therefore, $|\mathcal{F}| \leq n-1 + 1 = n$.
\end{proof}

\begin{all}
$g(n,2,3) \leq {n \choose 2-1} = n$.
\end{all}

\begin{proof}
Here we can assume that $n \geq 3$. If there exist $A$, $B \in \mathcal{F}$ such that $|A \cap B| = 0$, then for all $C \in \mathcal{F}$, $d(A,B,C) \leq 2$, so $|\mathcal{F}| \leq 2$. If $A \cap B \neq \emptyset$ for all $A$ and $B$ in $\mathcal{F}$, then fix $A = \{ a,b\}$ and $B= \{a,c\}$. Then for all $C \in \mathcal{F}$, either $a \in C$ or ${b,c} \in C$, so $|\mathcal{F}| \leq n-1+1=n$. 
\end{proof}

From here we prove both theorems by induction on $n$ and $k$. We have already proven the statement if either $k = 2$ or $n = 3k - \ell$. Now, let us consider the case where $k \geq 3$ and $n > 3k - \ell$.

According to Claim \ref{rekurzio} we know that $g(n,k,2) \leq g(n-1,k,2) + g(n-1,k-1,2) = {n \choose k-1}+{n-1 \choose k-2} + {n \choose k-2}+{n-1 \choose k-3} = {n+1 \choose k-1}+{n \choose k-2}$ and $g(n,k,3) \leq g(n-1,k,3) + g(n-1,k-1,3) = {n-1 \choose k-1} + 2 {n-1 \choose k-3} + 3 {n-2 \choose k-3} + {n-1 \choose k-2} + 2 {n-1 \choose k-4} + 3 {n-2 \choose k-4} = {n \choose k-1} + 2 {n \choose k-3} + 3 {n-1 \choose k-3}$, so the theorems are true. \end{proof}

\subsection{Lower bound for the uniform case}

\begin{mydef}
\label{elem}
Let $f(\mathcal{H},m)$ denote the number of elements of $[n]$ contained in exactly $m$ members of $\mathcal{H}$.
\end{mydef} 

This definition leads to the following easily verifiable statement:

\begin{all}
\label{elemenkent}
For all $\mathcal{A} \in \left( {2^{[n]}} \right)^3$ we have $d(\mathcal{A}) = 3 \cdot f(\mathcal{A},3) + f(\mathcal{A},2)$.    
\end{all}

Using this, we can give a lower bound on $d(\mathcal{A})$.

\begin{all}
\label{3k}
Let $\mathcal{A} \in \left( {2^{[n]}} \right)^3$ be a family and let the sum of their sizes be $s=\vert A \vert + \vert B \vert + \vert C \vert$. Now if $n \leq h \leq 2n$ then
\begin{equation} \label{n<h<2n}
    d(\mathcal{A}) \geq s-n
\end{equation}
if $2n \leq s \leq 3n$ then
\begin{equation} \label{2n<h<3n}
    d(\mathcal{A}) \geq n + 2 \cdot (s-2n)
\end{equation}
hold.
\end{all}

\begin{proof}
It is obvious that 
\begin{equation} \label{osszeg1}
s=3f(\mathcal{A},3)+2f(\mathcal{A},2)+f(\mathcal{A},1)
\end{equation}
and hence we have
\begin{equation} \label{elem1}
    d(\mathcal{A})=3f(\mathcal{A},3)+
    f(\mathcal{A},2)=s-f(\mathcal{A},2)-f(\mathcal{A},1).
\end{equation}
Here $f(\mathcal{A},2)+f(\mathcal{A},1)\leq n$ is obvious and \eqref{elem1} leads to \eqref{n<h<2n}. 

In order to prove \eqref{2n<h<3n} let us start with the inequality $3n \geq 2f(\mathcal{A},1) + 3f(\mathcal{A},2) + 3f(\mathcal{A},3)$. Using \eqref{osszeg1} the right-hand side can be written in the following form: $3n \geq f(\mathcal{A},1) + f(\mathcal{A},2)+s$.

This is equivalent to the inequality
\begin{equation} \label{elem2}
    s-f(\mathcal{A},1)-f(\mathcal{A},2) \geq n + 2(s-2n).
\end{equation}
Here the left hand side is just $d(\mathcal{A})$ by \eqref{elem1} and \eqref{elem2} really gives \eqref{2n<h<3n}. \end{proof}

\begin{all} \label{f(j,l)}
Let $j$ and $\ell \leq 3j$ be integers and let $A$, $B$, $C \in \displaystyle\binom{[f(j,\ell)]}{j}$ be three sets. Then $d(\mathcal{A}) \geq \ell.$
\end{all}

\begin{proof}
Case 1: $\frac{\ell}{3} \leq j < \frac{2\ell}{3}$. Using the notation of \ref{3k}, $s=3j$ and $n=f(j,l)=2j-\lceil \frac{\ell}{3} \rceil$.
Since $2(2j-\lceil \frac{\ell}{3} \rceil) \leq  3j \leq 3(2j-\lceil \frac{\ell}{3} \rceil)$ we can use \eqref{2n<h<3n}, so $d(\mathcal{A}) \geq 2s-3n = 2 \cdot 3j - 3(2j-\lceil \frac{\ell}{3} \rceil) \geq \ell$.

Case 2: $\frac{2\ell}{3} \leq j$. Using the same notation, $h=3j$ and $n=f(j,\ell)=3j-\ell$. Since $3j-\ell \leq  3j \leq 2(3j-\ell)$ we can use \eqref{n<h<2n}, so $d(\mathcal{A}) \geq s-n = 3j-(3j-\ell) = \ell$.
\end{proof}

\textit{Proof of Theorem \ref{unialso}.} Let $j$ be an integer and let $\mathcal{F} = \{ F : F \in {[n] \choose k}, |F \cap [f(j,\ell)]| \geq j\}$ be a family. Because of Claim \ref{f(j,l)} we know that $\mathcal{F} \in \mathbf{H}_\ell$. The expression ${f(j,\ell) \choose i}{n-f(j,\ell) \choose k-i}$ is the number of sets containing $i$ elements from $[3j-\ell]$. By choosing the optimal $j$ we get the lower bound. $\square$

\begin{mj} \label{triviuni}
If $3k \leq 2n$ and $\ell \leq 3k-n$ or $2n < 3k$ and $\ell \leq 9k-5n$, then $g(n,k,\ell) = {n \choose k}$.
\end{mj}

\begin{sej}
In cases that are not included in the Remark \ref{triviuni}, the lower bound in Theorem \ref{unialso} is sharp. 
\end{sej}

For a given $(n, k, \ell)$, it is difficult to compute the value \eqref{gnkl} in Theorem \ref{unialso}, as it must be determined for each $j$. In the rest of the section, we make observations about the expected optimal value of $j$.

\begin{sej} \label{ratio}
For every $\ell \geq 2$ there is a sequence $0 = \alpha_0(\ell) < \alpha_1(\ell) < \alpha_2(\ell) < \dots < \frac{1}{3}$ of real numbers such that if $n$, $k$ tends to infinity so that $\frac{k}{n} \rightarrow \gamma$ where $\alpha_{i-1}(\ell) < \gamma < \alpha_i(\ell)$ then the right-hand side of \eqref{gnkl} gives the maximum for $j=i$.
\end{sej}

\begin{mydef}
Let $m(n,k,\ell)$ be the minimum value of $j$ where \eqref{gnkl} gives the maximum.
\end{mydef}

In the following claim we determine the value of $\alpha_1(2)$ in Conjecture \ref{ratio}.

\begin{all}
Let $\ell=2$ and $n$ be sufficiently large. Then $m(n,k,\ell) \geq 2$ if the ratio $\frac{k}{n}$ is greater then $\frac{5-\sqrt{13}}{6}$.   
\end{all}

\begin{proof}
If $j=1$, then $f(j,\ell)=1$, so the value of the expression in \eqref{gnkl} is $n-1 \choose k-1$. If $j=2$, then $f(j,\ell)=4$, so in this case the value of the expression in \eqref{gnkl} is ${4 \choose 2}{n-4 \choose k-2}+{4 \choose 3}{n-4 \choose k-3}+{4 \choose 4}{n-4 \choose k-4}$. So $m(n,k,\ell) \geq 2$ if
$$\frac{(n-1)(n-2)(n-3)}{(k-3)(k-2)(k-1)(n-k-1)(n-k)} \leq 6 \cdot \frac{1}{(k-3)(k-2)}+ $$ $$4 \cdot \frac{1}{(k-3)(n-k-1)} + \frac{1}{(n-k-1)(n-k)}.$$

If we take $n$ (and $k$) to infinity, we get this approximation
$$\frac{n^3}{k^3(n-k)^2} \leq 6 \cdot \frac{1}{k^2}+4 \cdot \frac{1}{k(n-k)} + \frac{1}{(n-k)^2}.$$

Multiplying by the denominators it gives $n^3 \leq 6 \cdot k \cdot n^2 - 8 \cdot k^2 \cdot n + 3 \cdot k^3$. This means that $3 \left( \frac{k}{n} \right)^3-8 \left( \frac{k}{n} \right) ^2+6  \left( \frac{k}{n} \right) -1 \geq 0$. It is true, if  $1 \geq \frac{k}{n} \geq \frac{5-\sqrt{13}}{6}$ or $0 \geq \frac{k}{n}$, so $m(n,k,\ell) \geq 2$ if $\frac{k}{n} \geq \frac{5-\sqrt{13}}{6} \pm \varepsilon \approx 0.23$. 
\end{proof}

In a similar way we get that $m(n,k,\ell) \geq 3$ if and only if $\frac{k}{n} \geq \alpha_2 \approx 0.29$.

\textit{Proof of Theorem \ref{unialso2}.}
First let us assume that there are two sets $A$ and $B \in \mathcal{F}$ that are disjoint. Since $|A \cap B| = 0$, therefore $|A \cap C| + |B \cap C| \geq 2$, so all sets must contain at least 2 elements from $A \cup B$. An upper bound can be obtained by choosing the first two elements of the sets from $A \cup B$ and the others from $[n]$, so that $| \mathcal{F} | \leq {2k \choose 2}{n \choose k-2}$.

The other case when there are no two disjoint sets. In that case we can use Theorem \ref{1metszet}, so $|\mathcal{F}| \leq {n-1 \choose k-1}$ and equality is achieved if and only if there is an element $p \in [n]$ such that $H \in \mathcal{F}$ if and only if $p \in H$.

To prove the theorem we have to show that ${n-1 \choose k-1} \geq {2k \choose 2}{n \choose k-2}$. Simplifying this, we get that $(n-k+2)(n-k+1) \geq n k (k-1) (2k-1)$. Using the condition $n \geq 4 k^3$, we get that $(n-k+2)(n-k+1) \geq \frac{n^2}{2} \geq \frac{n}{2} 4k^3 \geq n k (k-1) (2k-1)$. $\square$

\section{The non-uniform case}

\subsection{General observations in the non-uniform case}

\begin{mydef}
Let $\mathcal{F} \subset 2^{n}$, $F \subset G$, $F \in \mathcal{F}$ and $G \notin \mathcal{F}$. We call $\mathcal{F} - \{ F \} \cup \{ G \}$ as an \textit{upward-shifting} of $\mathcal{F}$. 
A family $\mathcal{F'}$ is called an \textit{up-shifted} version of $\mathcal{F}$ if $\mathcal{F'}$ can be obtained from $\mathcal{F}$ by some upward-shifting.
\end{mydef}

\begin{mydef}
We call a family $\mathcal{F} \subset 2^{[n]}$ \textit{upward-closed} if for every $F \in \mathcal{F}$, if $F \subset G$, then $G \in \mathcal{F}$.
\end{mydef}

The following two claims are easy to prove.

\begin{all}
\label{nemuniupcl}
For each family $\mathcal{F} \subset 2^{[n]}$, there exists an upward-closed family $\mathcal{F'} \subset 2^{[n]}$ which is an up-shifted version of $\mathcal{F}$. $\square$
\end{all}

\begin{all}
\label{nemunimetszet}
Let $\mathcal{F}'$ be an up-shifted version of $\mathcal{F}$. Then if $d(\mathcal{A}) \geq \ell$ for all $\mathcal{A} \subset \mathcal{F}$, then $d(A',B',C') \geq \ell$ for all $\{A', B', C'\} = \mathcal{A'} \subset \mathcal{F}'$. $\square$
\end{all}

\begin{all}
Suppose that the family $\mathcal{F}$ is such that $d(\mathcal{A}) \geq \ell$ is true for all set-triples $\mathcal{A} \subset \mathcal{F}$ then there exists an upward-closed shifted family $\mathcal{F'}$ such that for all sets $\mathcal{A'} \subset \mathcal{F'}$, $d(\mathcal{A'}) \geq \ell$ and $|\mathcal{F}|=|\mathcal{F'}|$ are satisfied.
\end{all}

\begin{proof}
Take an upward-closed family $\mathcal{F}_1$ which is up-shifted of $\mathcal{F}$. It exists because of Claim \ref{nemuniupcl}. If it is a shifted family, then we are done. If not, then we can take a shifted family $\mathcal{F}_2$ which is a shift of $\mathcal{F}_1$ because of Claim \ref{unibalra}. If it is an upward-closed family, then we are done. If not, we can do this cycle again.

We can see that $\sum_{F \in \mathcal{F}} \vert F \vert$ increases with each upward-shifting and remains the same with each shift. Since $\sum_{F \in \mathcal{F}} \vert F \vert$ has an upper bound, so the cycle once it stops and then we get an upward-closed shifted family for which the condition is true because of Corollary \ref{balra} and Claim \ref{nemunimetszet}.
\end{proof}

\subsection{Sets with large intersection} 

\begin{jel}
Let $\overline{d(\mathcal{A})}=2 f(\mathcal{A},1) + 3 f(\mathcal{A},2) + 3 f(\mathcal{A},3)$ and $\overline{d(A,B)}=2 f(\{A,B\},1) + 3 f(\{A,B\},2)$. Denote by $\overline{\mathbf{H}_\ell}$ the class of families such that $\mathcal{H} \in \overline{\mathbf{H}_\ell}$ and $\mathcal{A} \subset \mathcal{H}$ imply $\overline{d(\mathcal{A})} \leq \ell$. Let $\mathcal{F}^-=\{ \overline{F} : F \in \mathcal{F} \}$ and let $\overline{h}(n,\ell) = \max \{ |\mathcal{F}| : \mathcal{F} \subset 2^{[n]}, \mathcal{F}\in \overline{\mathbf{H}_\ell}\}$.
\end{jel}

\begin{mj}
    The following proof is not working when $p=0$. It is easy to see that in this case $h(n,3n-x)=2$ if $0 \leq x \leq 3$ and $h(n, 3n-x)=3$ if $4 \leq x \leq 5$.
\end{mj}

\textit{Proof of Theorem \ref{3n-p}.} The following claims are needed to prove the theorem.

\begin{all} \label{kompl}
Let $A$, $B$, $C$ be sets such that $|A|+|B|+|C|=s$. Then $d(\mathcal{A}) + \overline{d(\mathcal{A})} = 2s$. 
\end{all}

\begin{proof}
    $d(\mathcal{A}) + \overline{d(\mathcal{A})} = 2 f(\mathcal{A},1) + 4 f(\mathcal{A},2) + 6 f(\mathcal{A},3) = 2s$.
\end{proof}

\begin{all}
$h(n,3n-\ell)= max \{ |\mathcal{F}| : \mathcal{F} \subset 2^{[n]}, \mathcal{F}\in \overline{\mathbf{H_{\ell}}}\}=\overline{h}(n,\ell)$.
\end{all}

\begin{proof} 
From Claim \ref{elemenkent} we know that $d(\mathcal{A}) = 3 f(\mathcal{A},3) + f(\mathcal{A},2)$ and $f(\mathcal{A},3) + f(\mathcal{A},2) + f(\mathcal{A},1) + f(\mathcal{A},0)=n$, so $d(\mathcal{A})= 3 n - 2 f(\mathcal{A},2) - 3 f(\mathcal{A},1) - 3  f(\mathcal{A},0) = 3n - 2f(\mathcal{A}^-,1) - 3f(\mathcal{A}^-,2) - 3f(\mathcal{A}^-,3)=3n - \overline{d(\mathcal{A}^-)}$.

We obtained that $d(\mathcal{A})+ \overline{d(\mathcal{A}^-)}=3n$ holds for all $\mathcal{A} \subset \mathcal{F}$ consequently $\mathcal{F} \in \mathbf{H_{\ell}}$ if and only if $\mathcal{F}^- \in \overline{\mathbf{H_{\ell}}}$.
\end{proof}

Using this claim, we obtain the following modified problem. Maximize $|\mathcal{F}|$ under the condition that all $\mathcal{A} \subset \mathcal{F}$ satisfy $\overline{d(\mathcal{A})} \leq x$.

We  will distinguish cases according to the divisibility of $x$ by 6. Let $x=6p+q$, where $p$, $q$ are integers and $0 \leq q \leq 5$.

\begin{all}
Let $x \geq 6p$. Then $\overline{h}(n,x) \geq \sum_{i=0}^{p} {n \choose i}$.
\end{all}

\begin{proof}
Let $\mathcal{F}= \{ F : |F| \leq p\}$ be a family. By Claim \ref{kompl}, for all $\mathcal{A} \subset \mathcal{F}$, $\overline{d(\mathcal{A})} \leq 2(|A|+|B|+|C|) = 6p$, which means $\mathcal{F}\in \overline{\mathbf{H}_{6p}} \subset \overline{\mathbf{H}_{x}}$.   
\end{proof}

\begin{all} \label{4p+q}
Let $x=6p+q$ and $n \geq 2^{3p+2} p^2 + p +1$. If there exists $A$, $B \in \mathcal{F}$ such that $\overline{d(A,B)} > 4p+q$, then $\mathcal{F}$ cannot be a maximal family in $\overline{\mathbf{H}_x}$.
\end{all}

\begin{proof}
Suppose that $\mathcal{F} \in \overline{\mathbf{H}_x}$ and $A$, $B \in \mathcal{F}$. We can see that $2 \cdot f(\{A,B\},1) + 2 \cdot f(\{A,B\},2) \leq \overline{d(A,B)} \leq 6p+q$ so $|A \cup B| \leq \lfloor \frac{6p+q}{2} \rfloor \leq 3p+2$.

If there exists $C \in \mathcal{F}$ such that $| ([n] \setminus A \setminus B) \cap C | \geq p$, then $\overline{d(\mathcal{A})} \geq \overline{d(A,B)} + 2p > 6p+q$ which is a contradiction. So for all $C \in \mathcal{F}$, $| ([n] \setminus A \setminus B) \cap C |< p-1$ holds, therefore $|\mathcal{F}| \leq 2^{|A \cup B|} \cdot \sum_{i=0}^{p-1} {|[n] \setminus A \setminus B| \choose i} \leq 2^{3p+2} \cdot \sum_{i=0}^{p-1} {n \choose p-1} < 2^{3p+2} p {n \choose p-1}$.

If $n \geq 2^{3p+2} p^2 + p +1$ then $|\mathcal{F}| < 2^{3p+2} p {n \choose p-1} \leq {n \choose p} < \sum_{i=0}^{p} {n \choose i}$, hence $\mathcal{F}$ cannot be a maximal family.
\end{proof}

\begin{all} \label{2p+q}
Let $x=6p+q$ and $n \geq 2^{3p+2} \cdot p^2 + p$. If there is an $A \in \mathcal{F} \in \overline{\mathbf{H}_x}$ such that $|A| > p+\frac{q}{2}$ then $\mathcal{F} < \overline{h}(n,x)$.
\end{all}

\begin{proof} 
If $|A| \geq 3p+3$, then $\mathcal{F} \notin \overline{\mathbf{H}_x}$, so we can assume that $|A| \leq 3p+2$. If there exists a $B \in \mathcal{F}$ such that $|([n] \setminus A) \cap B| \geq p$, then $\overline{d(A,B)} > 4p+q$, so we are done by Claim \ref{4p+q}.
    
If there is no such set, then $|\mathcal{F}| \leq 2^{3p+2} p {n \choose p-1} < \sum_{i=0}^{p} {n \choose i} \leq \overline{h}(n,x)$.
\end{proof}

Denote the family $\mathcal{F} \cap {[n] \choose i}$ by $\mathcal{F}_i$.

The proof distinguishes cases according to the value of $q$:

\textbf{Case 1: $q=0$ or $q=1$.} 
By Claim \ref{2p+q}, if $|\mathcal{F}| = \overline{h}(n,x)$, then all sets in $\mathcal{F}$ have at most $p$ elements. From here it is easy to see that  $|\mathcal{F}| \leq \sum_{i=0}^p {n \choose p}$ and $\mathcal{F}=\{F : |F| \leq p\}$ is a maximal family.

\textbf{Case 2: $q=2$.} 
By Claim \ref{2p+q}, if $|\mathcal{F}| = \overline{h}(n,x)$, then all sets in $\mathcal{F}$ have at most $p+1$ elements. 
If there are sets $A$ and $B \in \mathcal{F}$ such that $|A|=|B|=p+1$ and $|A \cap B| \leq 1$, then $\overline{d(A,B)} \geq 4p+3$, so $|\mathcal{F}| < \overline{h}(n,x)$ by Claim \ref{4p+q}.

Theorem \ref{tmetszet} implies $|\mathcal{F}_{p+1}| \leq {n-2 \choose p+1-2}$. So $|\mathcal{F}| \leq \sum_{i=0}^p {n \choose p} + {n-2 \choose p-1}$. The equality holds if $|F| \leq p$ or $|F|=p+1$ and $1$, $2 \in F$.

\textbf{Case 3: $q=4$.} By Claim \ref{2p+q}, if $|\mathcal{F}| = \overline{h}(n,x)$, then all sets in $\mathcal{F}$ have at most $p+2$ elements. 

If there exists $A \in \mathcal{F}$ such that $|A|=p+2$, then $|A \cap B| \geq 4$ for all $B \in \mathcal{F}_{p+2}$ and $|A \cap C| \geq 2$ for all $C \in \mathcal{F}_{p+1}$ by Claim \ref{4p+q}. So in this case $|\mathcal{F}_{p+1}| \leq {p+2 \choose 2} {n-2 \choose p-1}$ and $|\mathcal{F}_{p+2}| \leq {n-4 \choose p-2}$ by Theorem \ref{tmetszet}. Hence we have $|\mathcal{F}| \leq {n-4 \choose p-2} + {p+2 \choose 2} {n-2 \choose p-1} + \sum_{i=0}^p {n \choose p}$.

If there are no sets with $p+2$ elements, then all sets have at most $p+1$ elements. If there exist sets $A$, $B$, $C \in \mathcal{F}_{p+1}$ 
satisfying $d(\mathcal{A}) \leq 1$, then $\overline{d (\mathcal{A})} \geq 6p+5$ by Claim \ref{kompl}. So $\mathcal{F}_{p+1} \in \mathbf{H}_2$. Since $n \geq 2^{3p+2} p^2 + p + 1 \geq 4 (p+1)^3$, we can use Theorem \ref{unialso2}, so $\mathcal{F}$ contains at most ${n-1 \choose p}$ sets with $p+1$ elements so $|\mathcal{F}| \leq {n-1 \choose p} + \sum_{i=0}^p {n \choose p}$.

Comparing the two cases, we can see that the upper bound obtained in the latter case is larger, so $|\mathcal{F}| \leq {n-1 \choose p} + \sum_{i=0}^p {n \choose p}$ and equality holds if $\mathcal{F} = \{F: |F| \leq p$ or $(|F|=p+1$ and $1 \in F) \}$.

\textbf{Case 4: $q=3$.} Since $\mathcal{F} = \{F: |F| \leq p$ or ($|F|=p+1$ and $1 \in F) \} \in \overline{\mathbf{H}_{6p+3}}$, then ${n-1 \choose p} + \sum_{i=0}^p {n \choose p} \leq \overline{h}(n,6p+3)$.

Since $\mathcal{F} \in \overline{\mathbf{H}_{x-1}}$ implies $\mathcal{F} \in \overline{\mathbf{H}_x}$, then $\overline{h}(n,6p+3) \leq \overline{h}(n,6p+4) = {n-1 \choose p} + \sum_{i=0}^p {n \choose p} $.

\textbf{Case 5: $q=5$.} By Claim \ref{2p+q}, if $|\mathcal{F}| = \overline{h}(n,x)$, then all sets in $\mathcal{F}$ have at most $p+2$ elements. 

If there exist sets $A$, $B \in \mathcal{F}_{p+2}$ and a set $C \in \mathcal{F}_{p+1}$, then $|A \cap B| \geq 3$ and $|A \cap C| \geq 1$ by Claim \ref{4p+q}. 

There are 3 subcases we need to look at:

\begin{enumerate}
    \item There exist sets $A \in \mathcal{F}_{p+2}$ and $B \in \mathcal{F}_{p+1}$ such that $|A \cap B| = 1$.
    \item There exists a set $A \in \mathcal{F}_{p+2}$ but there are no sets $B \in \mathcal{F}_{p+1}$ such that $|A \cap B| \leq 1$.
    \item $\mathcal{F}_{p+2}$ is empty.
    
\end{enumerate}

(1): The family $\mathcal{F}_{p+2}$ forms a $3$-intersecting family, so by Theorem \ref{tmetszet}, $|\mathcal{F}_{p+2}| \leq {n-3 \choose p-1}$.

Since $|A \cap B|=1$, then $\overline{d(A,B)} = 4p+5$, so $|C \setminus (A \cup B)| \leq p-1$ true for all $C \in \mathcal{F}_{p+1}$. From this we get that $|\mathcal{F}_{p+1}| \leq {|(A \cup B)| \choose 2}{n \choose p-1} = {2p+3 \choose 2}{n \choose p-1}$.

(2): $|\mathcal{F}_{p+2}| \leq {n-3 \choose p-1}$ holds as before. Since $|A \cap B| \geq 2$ is true for all $B \in \mathcal{F}_{p+1}$, then $|\mathcal{F}_{p+1}| \leq {|A| \choose 2}{n \choose p-1} = {p+2 \choose 2}{n \choose p-1}$.

(3): If there are 3 pairwise disjoint members of $\mathcal{F}_{p+1}$, then $\overline{d(\mathcal{A})} \geq 6p+6$, so $d(\mathcal{A}) \geq 1$ true for all $\mathcal{A} \subset \mathcal{F}_{p+1}$. Since $n \geq 2^{3p+2} p^2 + p + 1 \geq 3 (p+1)^2 \cdot 2$, then $|\mathcal{F}_{p+1}| \leq \max \{ {n \choose p+1} - {n-2 \choose p+1}, {(p+1)(2+1)-1 \choose p+1} \}$ by using the special case of EMC \cite{Huang}. If $n \geq 7p$, then $\prod_{i=n-p}^{n-2} i \geq \prod_{i=2p+4}^{3p+2} i$ and $n (n-1) - (n-p-1)(n-p-2) \geq (2p+3)(2p+2)$, so the first value is the larger and $|\mathcal{F}_{p+1}| \leq {n \choose p+1} - {n-2 \choose p+1} = {n-1 \choose p} + {n-2 \choose p}$.

Summarizing (1), (2), (3): $|\mathcal{F}| \leq \max \{ {n-3 \choose p-1} + {2p+3 \choose 2}{n \choose p-1}, {n-3 \choose p-1} + {p+2 \choose 2}{n \choose p-1}, {n-1 \choose p} + {n-2 \choose p}\} + \sum_{i=0}^p {n \choose p}$. It can be proven by calculation that ${n-3 \choose p-1} \leq {n-2 \choose p}$ and ${p+2 \choose 2}{n \choose p-1} \leq {2p+3 \choose 2}{n \choose p-1} \leq {n-1 \choose p}$ if $n \geq 8p^3$, so $|\mathcal{F}| \leq {n-1 \choose p} + {n-2 \choose p} + \sum_{i=0}^p {n \choose p}$. Equality holds if $\mathcal{F} = \{F: |F| \leq p$ or $(|F|=p+1$ and $|\{1, 2\} \cap F| \geq 1) \}$. $\square$

\vspace{1cm}


\section*{Acknowledgement}

The author would like to thank Gyula Katona for his helpful remarks regarding this manuscript.


\begin{thebibliography}{99}

\bibitem{Bollobas}
B. Bollobás, D.E. Daykin, P. Erdős. Sets of independent edges of a hypergraph. Q. J. Math., 27 (N2) (1976), pp. 25-32

\bibitem{EKR}
P. Erdős, C. Ko, and R. Rado. Intersection theorems for systems of finite sets. The Quarterly Journal of Mathematics, 12(1):313–320, 1961.

\bibitem{ErdosMC}
P. Erdős. A problem on independent r-tuples. Ann. Univ. Sci. Budapest., 8 (1965), pages 93-95

\bibitem{ErdosGallai}
P. Erdős, T. Gallai, On maximal paths and circuits of graphs, Acta Math. Acad. Sci. Hung. 10 (1959) 337–356.

\bibitem{Frankl3}
P. Frankl. The Erdős-Ko-Rado theorem is true for n = ckt. In Combinatorics (Proceedings of the Fifth Hungarian Colloquium, Keszthely,
1976), volume 1, pages 365–375, 1978.

\bibitem{Frankl1}
P. Frankl. Multiply-Intersecting Families, Journal of combinatorial theory, Series B 53, 195-234 (1991)

\bibitem{Frankl2}
P. Frankl. Some exact results for multiply intersecting families, Journal of Combinatorial Theory, Series B

\bibitem{Frankl4}
P. Frankl, On the maximum number of edges in a hypergraph with a given matching number, Discrete Appl. Math. 216 (N3) (2017) 562–581.

\bibitem{Frankl5}
P. Frankl. Improved bounds for Erdős' Matching Conjecture, Journal of Combinatorial Theory, Series A 120 (5), 1068-1072.

\bibitem{Frankl6}
P. Frankl. The shifting technique in extremal set theory, Surveys in combinatorics 1987, 81-110.

\bibitem{FranklKup}
P. Frankl, A. Kupavskii. The Erdős Matching Conjecture and concentration inequalities, Journal of Combinatorial Theory, Series B 157, 366-400. (2021)


\bibitem{Frankl}
P. Frankl, A. Kupavskii. Two problems on matchings in set families - In the footsteps of Erdős and Kleitman, Journal of Combinatorial Theory, Series B (2019) 286-313.

\bibitem{FranklRodl}
P. Frankl, V. Rödl, A. Ruciński, On the maximum number of edges in a triple system not containing a disjoint family of a given size, Comb. Probab. Comput. 21 (1-2) (2012) 141–148.

\bibitem{ErdosMCcikk}
P. Frankl, A. Kupavskii. The Erdős Matching Conjecture and concentration inequalities. Journal of Combinatorial Theory, Series B
volume 157, 2022, pages 366-400

\bibitem{Huang}
H. Huang, P.-S. Loh, B. Sudakov. The size of a hypergraph and its matching number. Comb. Probab. Comput., 21 (N3) (2012), pp. 442-450

\bibitem{Katona}
G. O. H. Katona. Intersection theorems for systems of finite sets. Acta Mathematica Hungarica, 15(3-4):329–337, 1964.

\bibitem{Kleitman}
D.J. Kleitman. Maximal number of subsets of a finite set no k of which are pairwise disjoint, J. Combin. Theory 5 (1968) 157–163.

\bibitem{Luczak}
T. Łuczak, K. Mieczkowska, On Erdős’ extremal problem on matchings in hypergraphs, J. Comb. Theory, Ser. A 124 (2014) 178–194.

\bibitem{Wilson}
R. M. Wilson. The exact bound in the Erdős-Ko-Rado theorem. Combinatorica, 4(2-3):247–257, 1984.

\end{thebibliography}
\end{document}